\documentclass{amsart}
\usepackage{amssymb}
\usepackage{amsthm}
\newtheorem{theor}{Theorem}[section] \newtheorem{lem}{Lemma}[section]

\theoremstyle{definition} \newtheorem{defin}{Definition}[section]
\newtheorem{ex}{Example}[section]
\theoremstyle{remark} 
\newcommand{\pn}{\par\noindent} \newcommand{\pmn}{\par\medskip\noindent}

\begin{document}
\title{On enumeration of tree-rooted planar cubic maps}
\author{Yury Kochetkov}
\date{}
\begin{abstract} We consider planar cubic maps, i.e. connected cubic graphs
embedded into plane, with marked spanning tree and marked directed edge (not in
this tree). The number of such objects with $2n$ vertices is $C_{2n}\cdot
C_{n+1}$, where $C_k$ is Catalan number.\end{abstract}
\email{yukochetkov@hse.ru, yuyukochetkov@gmail.com} \maketitle

\section{Introduction}
\pn Plane triangulation is a planar map, where the perimeter of
each face is three. The corresponding dual graph is \emph{cubic},
i.e. the degree of each vertex is three. A plane triangulation
will be called \emph{proper}, if each edge is incident to exactly
\emph{two} faces. Otherwise it will be called \emph{improper}.

\begin{ex}
\[\begin{picture}(265,70) \put(15,5){\small proper triangulation}
\put(165,5){\small improper triangulation} \put(25,20){\line(1,0){60}}
\put(25,20){\circle*{2}} \put(85,20){\circle*{2}} \put(55,65){\circle*{2}}
\put(25,20){\line(2,3){30}} \put(85,20){\line(-2,3){30}}
\put(215,45){\oval(40,40)} \put(180,45){\circle*{2}} \put(195,45){\circle*{2}}
\put(210,45){\circle*{2}} \put(180,45){\line(1,0){30}}
\end{picture}\] The corresponding dual graphs are presented below:
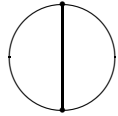
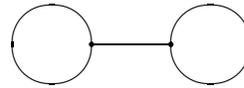
\[\begin{picture}(160,50) \put(20,25){\oval(40,40)} \put(20,5){\circle*{2}}
\put(20,45){\circle*{2}} \put(20,5){\line(0,1){40}} \put(60,22){¨}
\put(100,25){\oval(30,30)} \put(160,25){\oval(30,30)} \put(115,25){\circle*{2}}
\put(145,25){\circle*{2}} \put(115,25){\line(1,0){30}}
\end{picture}\]
\end{ex} \pn A connected graph with marked directed edge will be called
\emph{edge-rooted}. Proper edge-rooted triangulations where
enumerated by Tutte in the work \cite{Tu}: the number $T_n$ of
proper planar triangulations with $2n$ faces and marked directed
edge is
$$T_n=\frac{2\,(4n-3)!}{n!\,(3n-1)!}.$$ \pmn A combinatorial proof
of Tutte formula see in \cite{PS} (see also \cite{AP}). \pmn Let
$F_n$ be the number of planar edge-rooted cubic graphs with $2n$
vertices, i.e. the number of planar edge-rooted triangulations
(proper and improper) with $2n$ faces. Let us define numbers
$f_n$, $n\geqslant -1$, in the following way:
\begin{itemize}
    \item $f_{-1}=1/2$;
    \item $f_0=2$;
    \item $f_n=(3n+2)F_n$, $n>0$.
\end{itemize} In \cite{GJ} a recurrent relation for numbers $f_n$ was
proposed:
$$f_n=\frac{4(3n+2)}{n+1}\sum_{\scriptsize \begin{array}{c} i\geqslant -1,\, j\geqslant
-1\\ i+j=n-2\end{array}} f(i)f(j).\eqno(1)$$
\begin{ex} From (1) it follows that $F_1=4$. Indeed, there are four ways to choose a
root edge in a planar cubic map with two vertices:
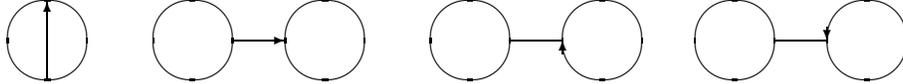
\[\begin{picture}(340,40) \put(15,20){\oval(30,30)}  \put(15,5){\vector(0,1){30}}
\put(70,20){\oval(30,30)} \put(120,20){\oval(30,30)}
\put(85,20){\vector(1,0){20}} \put(175,20){\oval(30,30)}
\put(225,20){\oval(30,30)} \put(190,20){\line(1,0){20}}
\put(210,15){\vector(0,1){5}} \put(275,20){\oval(30,30)}
\put(290,20){\line(1,0){20}} \put(325,20){\oval(30,30)}
\put(310,25){\vector(0,-1){5}} \end{picture}\] Also we have that $F_2=32$.
Indeed, there are six cubic maps with 4 vertices (and 6 edges):
\[\begin{picture}(300,50) \put(40,25){\oval(40,40)} \put(40,5){\line(0,1){40}}
\put(60,25){\line(1,0){20}} \put(95,25){\oval(30,30)} \put(5,20){\small 1)}

\put(155,25){\oval(30,30)} \put(215,25){\oval(30,30)}
\qbezier(155,40)(185,55)(215,40) \qbezier(155,10)(185,-5)(215,10)
\put(125,20){\small 2)}

\put(280,25){\oval(40,40)} \put(280,25){\line(0,1){20}} \put(245,20){\small 3)}
\put(280,25){\line(3,-2){17}} \put(280,25){\line(-3,-2){17}}  \end{picture}\]
\[\begin{picture}(310,50) \put(20,35){\circle{10}} \put(20,5){\circle{10}}
\put(40,20){\line(-4,3){16}} \put(40,20){\line(-4,-3){16}}
\put(40,20){\line(1,0){20}} \put(65,20){\circle{10}} \put(0,17){\small 4)}

\put(110,20){\oval(20,20)} \put(120,20){\line(1,0){20}}
\put(150,20){\oval(20,20)} \put(160,20){\line(1,0){20}}
\put(190,20){\oval(20,20)} \put(85,17){\small 5)}

\put(240,20){\oval(20,20)} \put(250,20){\line(1,0){20}}
\put(290,20){\oval(40,40)} \put(290,20){\oval(20,20)}
\put(300,20){\line(1,0){10}} \put(215,17){\small 6)}
\end{picture}\]
\begin{center}{Figure 1}\end{center}\pmn Group of automorphisms of the first
map is trivial, of the second has order 4, of the third has order
12, of the forth has order 3, of the fifth and the sixth has order
2. Thus, there are 12 ways to choose a root edge in the first map,
3 --- in the second, 1 --- in the third, 4
--- in the forth, 6 --- in the fifth and the sixth. All this gives
us 32 edge-rooted maps. \end{ex} \pn However, this formula does
not seem to have a geometrical/combinatorial explanation. \pmn In
\cite{Mu} a nice formula was proposed for the number tree-rooted
planar maps, i.e. edge-rooted planar maps with distinguished
spanning tree: the number of such maps with $n$ edges is $C_n\cdot
C_{n+1}$, where $C_k$ is $k$-th Catalan number. An elegant proof
of this formula see in \cite{Be}.

\begin{ex} There are four planar maps with two edges:
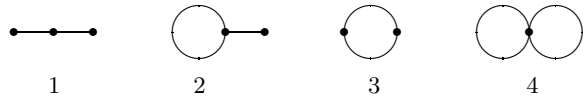
\[\begin{picture}(215,40) \multiput(0,25)(15,0){3}{\circle*{3}}
\put(0,25){\line(1,0){30}} \put(13,2){\small 1}

\put(70,25){\oval(20,20)} \put(80,25){\circle*{3}}
\put(95,25){\circle*{3}} \put(80,25){\line(1,0){15}}
\put(68,2){\small 2}

\put(125,25){\circle*{3}} \put(145,25){\circle*{3}}
\put(135,25){\oval(20,20)} \put(134,2){\small 3}

\put(195,25){\circle*{3}} \put(185,25){\oval(20,20)}
\put(205,25){\oval(20,20)} \put(194,2){\small 4}
\end{picture}\]
\begin{itemize}
    \item There is one way to choose a spanning tree in the first
    map and two ways to choose a directed edge.
    \item There is one way to choose a spanning tree in the second
    map and four ways to choose a directed edge.
    \item There is one way to choose a spanning tree in the third
    map and two ways to choose a directed edge.
    \item There is no spanning trees in the forth map and two ways
    to choose a directed edge.
\end{itemize} Thus we have $10=C_2\cdot C_3$ tree rooted planar
maps with two edges. \end{ex} \pn We will study tree-rooted cubic
maps with additional property: a root edge \emph{does not} belong
to the spanning tree. \pmn \textbf{Theorem.} \emph{The number of
such tree-rooted cubic maps with $2n$ vertices is $C_{2n}\cdot
C_{n+1}$, where $C_k$ is $k$-th Catalan number.}

\section{The main construction: from map to curve}

\begin{defin} By tree-rooted plane cubic map we will understand a
cubic graph imbedded into plane (sphere) with
\begin{itemize}
    \item marked spanning tree;
    \item marked directed edge that \emph{does not} belong to the
    spanning tree.
\end{itemize}
\end{defin}
\pn Let $G$ be a tree-rooted pane cubic map with $2n$ vertices. We
draw triangles, one triangle for each vertex, in such way that:
\begin{itemize}
    \item triangles are disjoint;
    \item each vertex is inside the corresponding triangle;
    \item each side of triangle intersect one outgoing edge of
    corresponding vertex.
\end{itemize}
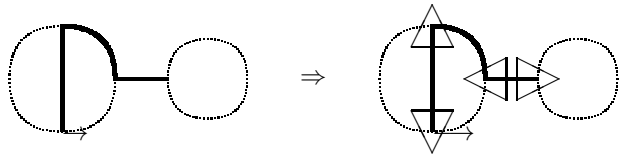
\begin{ex}
\[\begin{picture}(260,60) \qbezier[25](0,30)(0,50)(20,50)
\qbezier[25](0,30)(0,10)(20,10) \qbezier[25](20,10)(40,10)(40,30)
\qbezier[20](60,30)(60,45)(75,45)
\qbezier[20](60,30)(60,15)(75,15)
\qbezier[20](75,45)(90,45)(90,30)
\qbezier[20](75,15)(90,15)(90,30) \linethickness{0.5mm}
\put(20,10){\line(0,1){40}} \put(40,30){\line(1,0){20}}
\qbezier(20,50)(40,50)(40,30) \put(20,7){$\to$}
\put(110,28){$\Rightarrow$} \thinlines
\qbezier[25](140,30)(140,50)(160,50)
\qbezier[25](140,30)(140,10)(160,10)
\qbezier[25](160,10)(180,10)(180,30)
\qbezier[20](200,30)(200,45)(215,45)
\qbezier[20](200,30)(200,15)(215,15)
\qbezier[20](215,45)(230,45)(230,30)
\qbezier[20](215,15)(230,15)(230,30) \linethickness{0.5mm}
\put(160,10){\line(0,1){40}} \put(180,30){\line(1,0){20}}
\qbezier(160,50)(180,50)(180,30) \thinlines
\put(152,42){\line(1,0){16}} \put(152,42){\line(1,2){8}}
\put(168,42){\line(-1,2){8}} \put(152,18){\line(1,0){16}}
\put(152,18){\line(1,-2){8}} \put(168,18){\line(-1,-2){8}}
\put(188,22){\line(0,1){16}} \put(188,22){\line(-2,1){16}}
\put(188,38){\line(-2,-1){16}} \put(192,22){\line(0,1){16}}
\put(192,22){\line(2,1){16}} \put(192,38){\line(2,-1){16}}
\put(160,7){$\longrightarrow$}\end{picture}\] Thick lines above
mark spanning tree and an arrow indicates the direction of the
root edge. \end{ex} \pmn Two triangles will be called adjacent, if
the corresponding vertices are adjacent and the edge, that
connects them, belongs to the spanning tree.  The sides of
adjacent triangles that intersect this edge also will be called
adjacent. We construct a polygon $P$ by glewing adjacent triangles
by adjacent sides. This polygon has $2n+2$ sides and is divided
into $2n$ triangles. Each edge of the cubic map, that does not
belong to the spanning tree, intersects two sides of $P$ and we
will say that these sides constitute a pair. Polygon $P$ has a
marked side: the marked edge of the cubic map intersects it in
direction from inside $P$ to outside. \pmn Continuation of
Example.
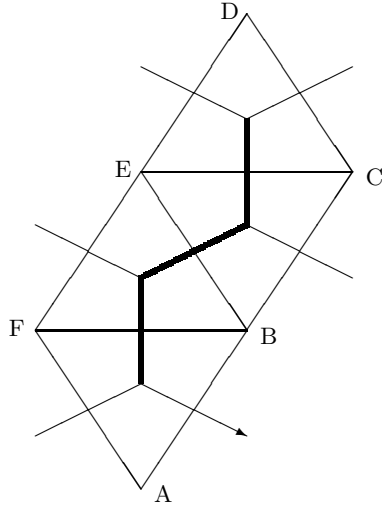
\[\begin{picture}(160,200) \put(20,70){\line(2,-3){40}}
\put(20,70){\line(1,0){80}} \put(20,70){\line(2,3){80}}
\put(60,10){\line(2,3){80}} \put(60,130){\line(2,-3){40}}
\put(60,130){\line(1,0){80}} \put(100,190){\line(2,-3){40}}
\linethickness{0.6mm} \put(60,50){\line(0,1){40}}
\qbezier(60,90)(80,100)(100,110) \put(100,110){\line(0,1){40}}
\thinlines \put(60,50){\line(-2,-1){40}}
\put(60,50){\vector(2,-1){40}} \put(60,90){\line(-2,1){40}}
\put(100,110){\line(2,-1){40}} \put(100,150){\line(2,1){40}}
\put(100,150){\line(-2,1){40}} \put(65,5){\small A}
\put(105,65){\small B} \put(145,125){\small C} \put(10,68){\small
F} \put(50,128){\small E} \put(90,188){\small D}
\end{picture}\] Here $EF$ and $FA$, $AB$ and $BC$, $CD$ and $DE$
are pairs and $AB$ is the marked edge. If we identify sides that
are in pairs (i.e. $EF$ with $FA$, $AB$ with $BC$ and $CD$ with
$DE$), then we will obtain a triangulated genus 0 curve.

\section{The main construction: from curve to map}
\pn Let $P$ be a $2n$-gon with marked side $M$ and triangulated by
non-intersecting diagonals into $2n-2$ triangles. Sides of $P$ are
divided into pairs in such way, that the identification of sides
in each pair gives us a genus 0 curve. We will construct a plane
tree-rooted cubic map with root edge (not in the spanning tree) in
the following way.
\begin{itemize}
    \item We put a vertex $v_i$ inside each triangle $\triangle_i$
    and connect vertices in adjacent triangles --- the
    spanning tree is constructed.
    \item Let sides $L$ and $L'$ be in pair. $L$ and $L'$ are sides
    of triangles $\triangle_i$ and $\triangle_j$, respectively
    (these triangles may coincide). We draw an arc that connect
    $v_i$ and $v_j$ in the following way: going from $v_i$ the arc
    intersects $L$. Its next part lies in the exterior of $P$ and connects
    $L$ and $L'$. After intersecting $L'$ the arc goes to $v_j$.
    \item An arc, that intersects $M$ will be the root edge. At
    intersection point it is directed from inside $P$ to outside.
\end{itemize}
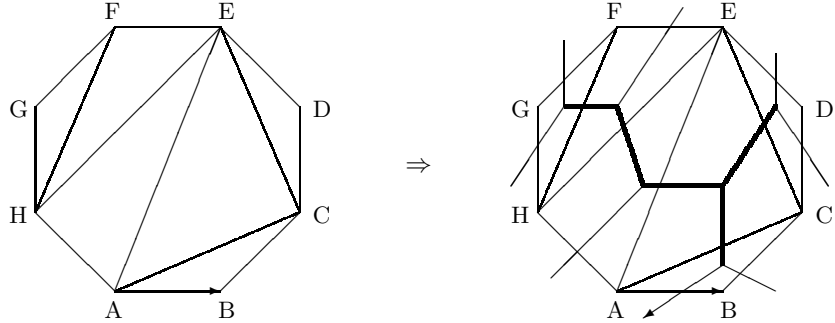
\begin{ex}
\[\begin{picture}(320,120) \put(10,40){\line(0,1){40}}
\put(10,40){\line(1,-1){30}} \put(10,40){\line(1,1){70}}
\put(10,80){\line(1,1){30}} \qbezier(10,40)(25,75)(40,110)
\put(40,110){\line(1,0){40}} \put(40,10){\vector(1,0){40}}
\put(40,10){\line(2,5){40}} \qbezier(40,10)(75,25)(110,40)
\put(80,110){\line(1,-1){30}} \qbezier(80,110)(95,75)(110,40)
\put(110,40){\line(0,1){40}} \put(80,10){\line(1,1){30}}
\put(0,36){\small H} \put(0,76){\small G} \put(36,0){\small A}
\put(36,113){\small F} \put(79,0){\small B} \put(79,113){\small E}
\put(115,36){\small C} \put(115,76){\small D}
\put(150,55){$\Rightarrow$}

\put(200,40){\line(0,1){40}} \put(200,40){\line(1,-1){30}}
\put(200,40){\line(1,1){70}} \put(200,80){\line(1,1){30}}
\qbezier(200,40)(215,75)(230,110) \put(230,110){\line(1,0){40}}
\put(230,10){\vector(1,0){40}} \put(230,10){\line(2,5){40}}
\qbezier(230,10)(265,25)(300,40) \put(270,110){\line(1,-1){30}}
\qbezier(270,110)(285,75)(300,40) \put(300,40){\line(0,1){40}}
\put(270,10){\line(1,1){30}} \put(190,36){\small H}
\put(190,76){\small G} \put(226,0){\small A} \put(226,113){\small
F} \put(269,0){\small B} \put(269,113){\small E}
\put(305,36){\small C} \put(305,76){\small D}

\linethickness{0.5mm} \put(210,80){\line(1,0){20}}
\qbezier(230,80)(235,65)(240,50) \put(240,50){\line(1,0){30}}
\put(270,50){\line(0,-1){30}} \qbezier(270,50)(280,65)(290,80)

\thinlines \put(270,20){\vector(-3,-2){30}}
\put(270,20){\line(2,-1){20}} \put(290,80){\line(2,-3){20}}
\put(290,80){\line(0,1){20}} \put(240,50){\line(-1,-1){35}}
\put(230,80){\line(2,3){25}} \put(210,80){\line(-2,-3){20}}
\put(210,80){\line(0,1){25}}
\end{picture}\]  Here sides $AB$ and $FG$, $BC$ and $CD$, $DE$ and
$EF$, $GH$ and $HA$ constitute pairs and $AB$ is the marked side.
Thus, we must connect the arc that intersects $AB$ with the arc
that intersects $FG$, the arc that intersects $BC$ with the arc
that intersects $CD$, the arc that intersects $DE$ with the arc
that intersects $EF$ and the arc that intersects $GH$ with the arc
that intersects $HA$. An arrow in the arc that intersects $AB$
indicates the direction of the root edge of the cubic graph. The
cubic graph itself and its "simplification" are presented in the
figure below.
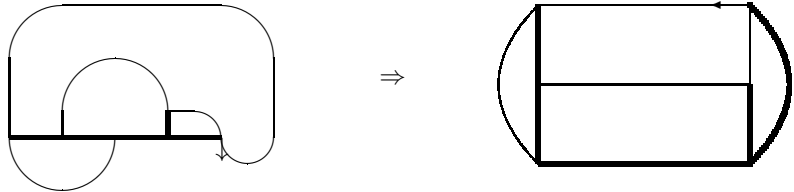
\[\begin{picture}(300,100) \linethickness{0.6mm} \put(0,30){\line(1,0){80}}
\put(60,30){\line(0,1){10}} \thinlines \put(20,30){\line(0,1){10}}
\put(20,30){\oval(40,40)[b]} \put(50,30){\oval(100,100)[t]}
\put(40,40){\oval(40,40)[t]} \put(60,30){\oval(40,20)[tr]}
\put(90,30){\oval(20,20)[b]} \put(78,23){$\downarrow$}
\put(140,50){$\Rightarrow$} \linethickness{0.6mm}
\put(200,20){\line(0,1){60}} \put(200,20){\line(1,0){80}}
\put(280,20){\line(0,1){30}} \qbezier(280,20)(310,50)(280,80)
\thinlines \qbezier(200,20)(170,50)(200,80)
\put(200,80){\line(1,0){80}} \put(200,50){\line(1,0){80}}
\put(280,50){\line(0,1){30}} \put(280,80){\vector(-1,0){15}}
\end{picture}\]
\end{ex}
\begin{lem} We can draw above mentioned arcs in such way, that
they do not intersect in the exterior of $P$.\end{lem}
\begin{proof} Let us connect midpoints of all sides in pairs by
segments inside $P$. As the identification of sides in pairs
generates a genus zero curve, then these segments do not
intersect. The polygon $P$ is embedded into sphere, so we can
interchange its interior and exterior domains. \end{proof}

\section{Main statement}
\pn
\begin{theor} The number of tree-rooted cubic maps with
$2n$ vertices and a marked edge, that does not belong to the
spanning tree, is $C_{2n}\cdot C_{n+1}$, where $C_k$ is $k$-th
Catalan number. \end{theor}
\begin{proof} Our theorem follows from two statements.
\begin{enumerate}
    \item A convex $n$-gon with a marked side can be divided into
    triangles by non-intersecting diagonals in $C_{n-2}$ ways \cite{St}.
    \item There are $C_n$ ways to define a pairwise identification
    of sides of a convex $2n$-gon with a marked side to obtain
    a genus 0 curve \cite{LZ}.
\end{enumerate}
\end{proof}

\begin{ex} According to theorem, we have $C_4\cdot C_3=70$ tree-rooted
cubic maps with $4$ vertices. In what follows a map with a marked
spanning tree will be called \emph{t-map}. The first cubic map in
Figure 1 generates six t-maps.
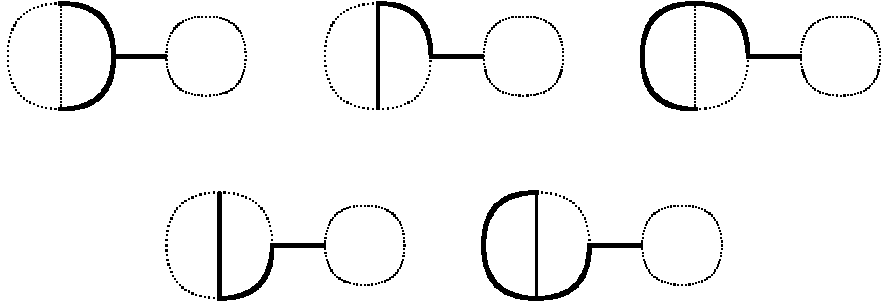
\[\begin{picture}(330,50) \qbezier[20](0,25)(0,45)(20,45)
\qbezier[20](0,25)(0,5)(20,5) \qbezier[20](60,25)(60,40)(75,40)
\qbezier[20](60,25)(60,10)(75,10)
\qbezier[20](75,40)(90,40)(90,25)
\qbezier[20](75,10)(90,10)(90,25) \qbezier[30](20,5)(20,25)(20,45)

\qbezier[20](180,25)(180,40)(195,40)
\qbezier[20](180,25)(180,10)(195,10)
\qbezier[20](195,40)(210,40)(210,25)
\qbezier[20](195,10)(210,10)(210,25)
\qbezier[20](120,25)(120,45)(140,45)
\qbezier[20](120,25)(120,5)(140,5)
\qbezier[20](140,5)(160,5)(160,25)

\qbezier[20](300,25)(300,40)(315,40)
\qbezier[20](300,25)(300,10)(315,10)
\qbezier[20](315,40)(330,40)(330,25)
\qbezier[20](315,10)(330,10)(330,25)
\qbezier[20](260,5)(280,5)(280,25)
\qbezier[30](260,5)(260,25)(260,45)

\linethickness{0.5mm} \qbezier(20,45)(40,45)(40,25)
\qbezier(20,5)(40,5)(40,25) \put(40,25){\line(1,0){20}}
\put(160,25){\line(1,0){20}} \put(280,25){\line(1,0){20}}
\qbezier(140,45)(160,45)(160,25) \put(140,5){\line(0,1){40}}
\qbezier(240,25)(240,45)(260,45) \qbezier(240,25)(240,5)(260,5)
\qbezier(260,45)(280,45)(280,25)
\end{picture}\]

\[\begin{picture}(210,50) \qbezier[20](0,25)(0,45)(20,45)
\qbezier[20](0,25)(0,5)(20,5) \qbezier[20](60,25)(60,40)(75,40)
\qbezier[20](60,25)(60,10)(75,10)
\qbezier[20](75,40)(90,40)(90,25)
\qbezier[20](75,10)(90,10)(90,25)
\qbezier[20](20,45)(40,45)(40,25)

\qbezier[20](180,25)(180,40)(195,40)
\qbezier[20](180,25)(180,10)(195,10)
\qbezier[20](195,40)(210,40)(210,25)
\qbezier[20](195,10)(210,10)(210,25)
\qbezier[20](140,45)(160,45)(160,25) \put(140,5){\line(0,1){40}}

\qbezier[20](140,5)(160,5)(160,25)
\qbezier[30](140,5)(140,25)(140,45)

\linethickness{0.5mm} \put(20,5){\line(0,1){40}}
\put(40,25){\line(1,0){20}} \qbezier(20,5)(40,5)(40,25)
\qbezier(120,25)(120,45)(140,45) \qbezier(120,25)(120,5)(140,5)
\put(160,25){\line(1,0){20}} \qbezier(140,5)(160,5)(160,25)
\end{picture}\] In each case we have six ways to choose a marked
edge, that does not belong to the tree. Thus, the first map
generates $30$ tree-rooted maps. \pmn The second cubic map in
Figure 1 generates four t-maps.
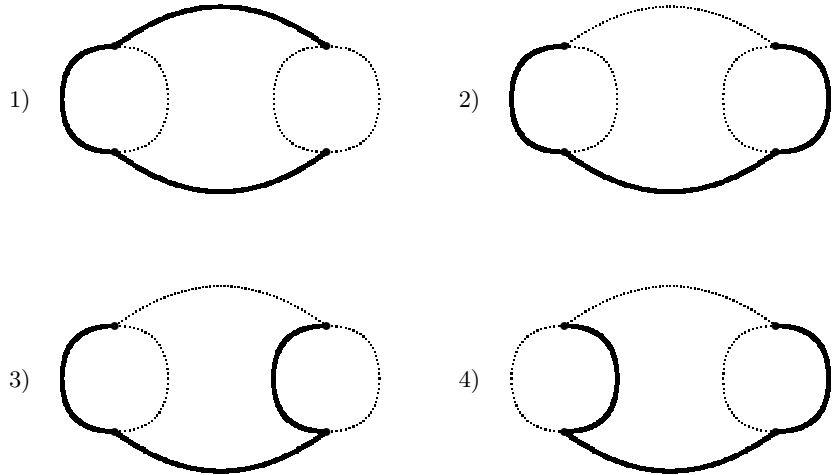
\[\begin{picture}(310,80) \put(0,37){\small 1)} \qbezier[20](40,60)(60,60)(60,40)
\qbezier[20](40,20)(60,20)(60,40)
\qbezier[20](120,60)(100,60)(100,40)
\qbezier[20](120,20)(100,20)(100,40)
\qbezier[20](120,60)(140,60)(140,40)
\qbezier[20](120,20)(140,20)(140,40) \put(40,20){\circle*{3}}
\put(40,60){\circle*{3}} \put(120,20){\circle*{3}}
\put(120,60){\circle*{3}}

\put(170,37){\small 2)} \qbezier[20](210,60)(230,60)(230,40)
\qbezier[20](210,20)(230,20)(230,40)
\qbezier[20](290,60)(270,60)(270,40)
\qbezier[20](290,20)(270,20)(270,40) \put(210,20){\circle*{3}}
\put(210,60){\circle*{3}} \put(290,20){\circle*{3}}
\put(290,60){\circle*{3}} \qbezier[60](210,60)(250,90)(290,60)

\linethickness{0.5mm} \qbezier(40,60)(20,60)(20,40)
\qbezier(40,20)(20,20)(20,40) \qbezier(40,20)(80,-10)(120,20)
\qbezier(40,60)(80,90)(120,60) \qbezier(210,60)(190,60)(190,40)
\qbezier(210,20)(190,20)(190,40) \qbezier(210,20)(250,-10)(290,20)

\qbezier(290,60)(310,60)(310,40) \qbezier(290,20)(310,20)(310,40)
\end{picture}\]

\[\begin{picture}(310,80) \put(0,37){\small 3)} \qbezier[20](40,60)(60,60)(60,40)
\qbezier[20](40,20)(60,20)(60,40)
\qbezier[20](120,60)(140,60)(140,40)
\qbezier[20](120,20)(140,20)(140,40)  \put(40,20){\circle*{3}}
\put(40,60){\circle*{3}} \put(120,20){\circle*{3}}
\put(120,60){\circle*{3}} \qbezier[60](40,60)(80,90)(120,60)

\put(170,37){\small 4)} \qbezier[20](210,60)(190,60)(190,40)
\qbezier[20](210,20)(190,20)(190,40)
\qbezier[20](290,60)(270,60)(270,40)
\qbezier[20](290,20)(270,20)(270,40) \put(210,20){\circle*{3}}
\put(210,60){\circle*{3}} \put(290,20){\circle*{3}}
\put(290,60){\circle*{3}} \qbezier[60](210,60)(250,90)(290,60)

\linethickness{0.5mm} \qbezier(40,60)(20,60)(20,40)
\qbezier(40,20)(20,20)(20,40) \qbezier(40,20)(80,-10)(120,20)
\qbezier(120,60)(100,60)(100,40) \qbezier(120,20)(100,20)(100,40)

\qbezier(210,60)(230,60)(230,40) \qbezier(210,20)(230,20)(230,40)
\qbezier(210,20)(250,-10)(290,20) \qbezier(290,60)(310,60)(310,40)
\qbezier(290,20)(310,20)(310,40)
\end{picture}\] The first two t-maps have trivial groups of
automorphisms. Thus, they generate six tree-rooted cubic maps
each. But the group of automorphisms of the third and the of forth
t-maps has order two. Thus, they generate three tree-rooted cubic
maps each and the second cubic map generates $18$ tree-rooted
maps. \pmn The third cubic map in Figure 1 generates three t-maps.
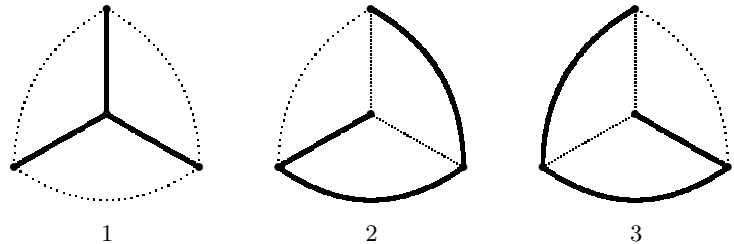
\[\begin{picture}(270,95) \put(0,30){\circle*{3}} \put(70,30){\circle*{3}}
\put(35,50){\circle*{3}} \put(35,90){\circle*{3}}
\qbezier[30](0,30)(35,5)(70,30) \qbezier[30](0,30)(0,70)(35,90)
\qbezier[30](70,30)(70,70)(35,90) \put(33,2){\small 1}

\put(100,30){\circle*{3}} \put(170,30){\circle*{3}}
\put(135,50){\circle*{3}} \put(135,90){\circle*{3}}
\qbezier[30](100,30)(100,70)(135,90)
\qbezier[30](135,50)(135,70)(135,90)
\qbezier[30](135,50)(152,40)(170,30) \put(133,2){\small 2}

\put(200,30){\circle*{3}} \put(270,30){\circle*{3}}
\put(235,50){\circle*{3}} \put(235,90){\circle*{3}}
\qbezier[30](270,30)(270,70)(235,90)
\qbezier[30](235,50)(235,70)(235,90)
\qbezier[30](200,30)(217,40)(235,50) \put(233,2){\small 3}

\linethickness{0.5mm} \qbezier(0,30)(17,40)(35,50)
\qbezier(35,50)(52,40)(70,30) \put(35,50){\line(0,1){40}}

\qbezier(100,30)(117,40)(135,50) \qbezier(100,30)(135,5)(170,30)
\qbezier(170,30)(170,70)(135,90)

\qbezier(200,30)(235,5)(270,30) \qbezier(200,30)(200,70)(235,90)
\qbezier(235,50)(252,40)(270,30)
\end{picture}\] The group of automorphisms of the first of them has
order $3$ and of the second and the third --- order $2$. Thus they
generate $2+3+3=8$ tree-rooted cubic maps. \pmn The forth cubic
map in Figure 1 generates one t-map with order three group of
automorphisms. Thus, it generates $2$ tree-rooted cubic maps. \pmn
The fifth cubic map in Figure 1 also generates one t-map with
trivial group of automorphisms. Thus, it generates $6$ tree-rooted
cubic maps. \pmn The sixth cubic map in Figure 1 generates two
t-maps
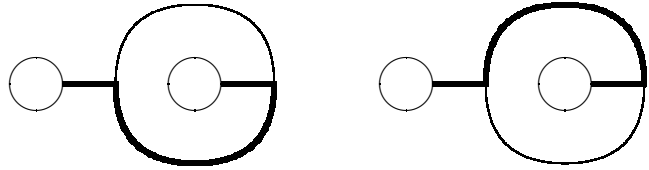
\[\begin{picture}(240,80) \put(10,40){\oval(20,20)}
\qbezier(40,40)(40,70)(70,70) \qbezier(70,70)(100,70)(100,40)
\put(70,40){\oval(20,20)}

\put(150,40){\oval(20,20)} \qbezier(180,40)(180,10)(210,10)
\qbezier(210,10)(240,10)(240,40) \put(210,40){\oval(20,20)}

\linethickness{0.6mm} \put(20,40){\line (1,0){20}}
\put(80,40){\line(1,0){20}} \qbezier(40,40)(40,10)(70,10)
\qbezier(70,10)(100,10)(100,40)

\put(160,40){\line (1,0){20}} \put(220,40){\line(1,0){20}}
\qbezier(180,40)(180,70)(210,70) \qbezier(210,70)(240,70)(240,40)
\end{picture}\] with order two group of
automorphisms each. Thus, they generate $3+3=6$ tree-rooted cubic
maps. \pmn So, we have $$30+18+8+2+6+6=70$$ tree-rooted cubic
maps, as expected.
\end{ex}

\vspace{5mm}

\begin{thebibliography}{99}
\bibitem{AP} Albenque M. and Poulalhon D., \emph{Generic method
for bijections between blossoming trees and planar maps}, arXiv:
1305.1312.
\bibitem{Be} Bernardi O., \emph{Bijective counting of tree-rooted
maps and shuffles of parenthesis systems}, arXiv:math/0601684.
\bibitem{GJ} Goulden I.P. and Jackson D.M., \emph{The KP hierarchy,
branched covers, and triangulations}, arXiv: 0803.3980.
\bibitem{LZ} Lando S. and Zvonkin A., \emph{Graphs on surfaces and
their applications. Encyclopedia of mathematical sciences}
\textbf{141}, Springer-Verlag, Berlin, 2004.
\bibitem{Mu} Mullin R.C., \emph{On the enumeration of tree-rooted
maps}, Canad. J. Math., 1967, 19, 174-183.
\bibitem{PS} Poulalhon D. and Schaeffer G., \emph{Optimal coding
and sampling of triangulations}, Algorithmica, 2006, 46(3),
505-527.
\bibitem{St} Stanley R.P., \emph{Enumerative combinatorics, volume
2}, Wadsworth \& Brooks, 1999.
\bibitem{Tu} Tutte W.T., \emph{A census of planar triangulations},
Canad. J. Math., 1962, 14, 21-38.
\end{thebibliography}
\end{document}